\documentclass[11pt, letterpaper]{article}

\usepackage{af_eq}
\usepackage{verbatim}

\newcommand{\dset}{\bm{X}}

\newcommand{\ours}{\mrm{Ours}}

\DeclareMathOperator*{\stabRank}{\mrm{r}}

\DeclareMathOperator*{\med}{\mrm{Med}}
\DeclareMathOperator*{\mad}{\mrm{MAD}}

\def\com{1}

\if\com1
\newcommand{\Ynote}[1]{\footnote{\color{ForestGreen}Yeshwanth: #1}}
\newcommand{\Znote}[1]{\footnote{\color{Orange}Zihao: #1}}
\else
\newcommand{\Ynote}[1]{}
\newcommand{\Znote}[1]{}
\fi

\newcommand{\yeshwanth}[1]
{
    \if\com1
        \todo[inline,color=green!30]{\small\textbf{Yeshwanth:} #1}
    \else
    \fi
}
\newcommand{\zihao}[1]
{
    \if\com1
        \todo[inline,color=red!30]{\small\textbf{Zihao:} #1}
    \else
    \fi
}

\begin{document}
\title{Statistical Barriers to Affine-equivariant Estimation}
\author{Zihao Chen\thanks{Department of Electrical Engineering and Computer Science, UC Berkeley. \texttt{zihaochen@berkeley.edu}} \and Yeshwanth Cherapanamjeri\thanks{CSAIL, Massachusetts Institute of Technology. \texttt{yesh@mit.edu}}}

\date{}

\maketitle

\begin{abstract}
    We investigate the quantitative performance of affine-equivariant estimators for robust mean estimation. As a natural stability requirement, the construction of such affine-equivariant estimators has been extensively studied in the statistics literature. We quantitatively evaluate these estimators under two outlier models which have been the subject of much recent work: the heavy-tailed and adversarial corruption settings. We establish lower bounds which show that affine-equivariance induces a \emph{strict} degradation in recovery error with quantitative rates degrading by a factor of $\sqrt{d}$ in both settings. We find that classical estimators such as the Tukey median {(Tukey '75)} and Stahel-Donoho estimator {(Stahel '81 and Donoho '82)} are either quantitatively sub-optimal \emph{even within} the class of affine-equivariant estimators or lack any quantitative guarantees. On the other hand, recent estimators with strong quantitative guarantees are not affine-equivariant or require additional distributional assumptions to achieve it. We remedy this by constructing a new affine-equivariant estimator which nearly matches our lower bound. Our estimator is based on a novel notion of a high-dimensional median which may be of independent interest. Notably, our results are applicable more broadly to \emph{any} estimator whose performance is evaluated in the \emph{Mahalanobis} norm which, for affine-equivariant estimators, corresponds to an evaluation in \emph{Euclidean} norm on isotropic distributions.
\end{abstract}

\thispagestyle{empty}
\setcounter{page}{0}
\newpage

\section{Introduction}
\label{sec:intro}

The design of estimators, robust to gross outliers in the data, has remained a mainstay of statistical research for the past 60 years \cite{huber64,htsurvey,diakonikolas2019recent}. A natural consideration of such a procedure is its stability under natural transformations of the data. Focusing on \emph{affine} transformations in particular and the fundamental estimation problem of mean estimation, several early results constructed estimators \emph{equivariant} to affine transformations of the data and studied the degree of their resilience to outliers \cite{tukeypicturing75,mrobustmest76,stahel1981,donoho1982breakdown,oja83,rousseeuwyohai84,rousseeuw85,hub85projection,liu92,dsbreakdown92}. Unfortunately, while these estimators guarantee \emph{finite} error, a sharp \emph{quantitative} understanding of their performance remains elusive. On the other hand, recent estimators \cite{lugosi2017sub,lmtrimmed} with quantitatively \emph{optimal} performance do not satisfy this natural stability property.

In this work, we quantitatively investigate the performance of affine-equivariant estimators. We find that some statistical degradation is \emph{necessary} in this setting with optimal rates degrading by a $\sqrt{d}$ factor. However, classical estimators are quantitatively \emph{sub-optimal} even within this restricted class. To remedy this, we design a novel affine-equivariant estimator with near-\emph{optimal} statistical performance and robustness. Our estimator is based on a novel notion of a high-dimensional median which may be of independent interest.

Formally, we study the robust mean estimation problem where we are given $n$ independent and identically distributed (i.i.d.) data points $X = \{X_i\}_{i = 1}^n \subset \R^d$ drawn from a distribution, $D$, with mean $\mu$ and variance $\Sigma$ along with a target failure probability $\delta$. Furthermore, an arbitrarily chosen $\eta$ fraction of the data points may be corrupted in a possibly adversarial way. The goal, now, is to design an estimator $\wh{\mu}$ with the smallest $r_\delta$ satisfying:
\begin{equation*}
    \P \lbrb{\norm{\wh{\mu} (X) - \mu}_\Sigma \leq r_\delta} \geq 1 - \delta \text{ where } \norm{x}_\Sigma \coloneqq \sqrt{x^\top \Sigma^{-1} x}.
\end{equation*}
The above notion of error, commonly referred to as the Mahalanobis Distance, is a natural affine-equivariant metric. Equivalently, the Mahalanobis distance may be viewed as measuring the \emph{Euclidean} distance under the linear transformation that renders the distribution isotropic. Hence, for affine-equivariant estimators, our results characterize the optimal achievable \emph{Euclidean} error on distributions with $\Sigma \preccurlyeq I$\footnote{Note that without any restriction on $\Sigma$, no finite error bound is possible.}. We present our results in Mahalanobis norm as our (lower) bounds hold for \emph{any} estimator whose performance is measured in this norm. Note, furthermore, that in this formulation, we make no other assumptions on the data distribution beyond the existence of a mean and variance, hence, allowing for heavy-tailed scenarios where higher-order moments might not even \emph{exist} and $\Sigma$ might not even be \emph{estimable} from the given samples.

As a point of comparison, we briefly overview the Euclidean setting where the error is measured in the \emph{Euclidean} norm and which has been the subject of much recent attention leading to optimal quantitative characterizations. Here, one aims to design an estimator minimizing $r_\delta$ satisfying:
\begin{equation*}
    \Pr \lbrb{\norm{\wh{\mu} (X) - \mu} \leq r_\delta} \geq 1 - \delta \text{ where } \norm{x} = \sqrt{x^\top x}. 
\end{equation*}
That is, one aims to find an estimator with the smallest possible \emph{Euclidean} confidence interval as a function of $n$, $d$, and $\delta$. Surprisingly, this fundamental question was only resolved recently in the pioneering work of Lugosi and Mendelson \cite{lugosi2017sub}, who provided the following characterization:
\begin{equation*}
    r_\delta = O \lprp{\sqrt{\norm{\Sigma}} \cdot \lprp{\sqrt{\frac{\stabRank (\Sigma) + \log (1 / \delta)}{n}} + \sqrt{\eta}}} \text{ where } \stabRank (\Sigma) \coloneqq \frac{\Tr \Sigma}{\norm{\Sigma}}.
\end{equation*}
When $\eta = 0$, this rate, referred to as the \emph{sub-Gaussian} rate, is known to be \emph{optimal} for Gaussians and hence, cannot be improved upon in general. This striking result holds despite making no other assumptions on the data distribution. Note that when $\Sigma \preccurlyeq I$, the above rate simplifies to: 
\begin{equation*}
    r_\delta = O \lprp{\sqrt{\frac{d + \log (1 / \delta)}{n}} + \sqrt{\eta}}.
\end{equation*}
However, for the \emph{Mahalanobis} norm, all known estimators require stronger assumptions to establish quantitative guarantees. Often, these results require the additional property that a multiplicative approximation to $\Sigma$ is estimable from the samples \cite{lmdirectiondependent20,bgsuz,dlrobsub22,dhk23,bhs23}.

Our upper bound remedying these difficulties is presented in the following theorem:
\begin{theorem}
    \label{thm:ub_afeq}
    There exist absolute constants $C_1, C_2 > 0$ such that the following hold. Let $n, d \in \N,  \delta \in (0, 1)$ and $\eta \in [0, 1 / (6d)]$. Suppose $\dset = \{X_i\}_{i = 1}^n$ are generated i.i.d. from a distribution $D$ with mean $\mu$ and covariance $\Sigma$. Then, there exists an affine-equivariant estimator, $\wh{\mu}$, which when given any $\eta$-corrupted version of $\bm{X}$ satisfies:
    \begin{equation*}
        \norm{\wh{\mu} (\dset) - \mu}_\Sigma \leq C_1 \lprp{\sqrt{\frac{d \log (1 / \delta)}{n}} + \sqrt{d\eta}}
    \end{equation*}
    with probability at least $1 - \delta$ over $\bm{X}$ when $n \geq C_2 d \log (1 / \delta)$.
\end{theorem}

Furthermore, we exhibit lower bounds establishing that the above rate and restrictions on $\eta$ are essentially tight. Our lower bounds are proved separately for the heavy-tailed and adversarial settings and hence, our upper bound which holds for \emph{both} settings simultaneously is optimal. We start with a lower bound for the heavy-tailed setting.
\begin{theorem}
    \label{thm:lb_ht_afeq}
    There exist absolute constants, $C_1, C_2, c > 0$ such that the following holds. Let $n, d \in \N$ and $\delta \in (0, 1)$ be such that $n \geq C_1 d \log (1 / \delta)$ and $\log (1 / \delta) \geq C_2 \log (2d)$. Then, there exists a family of distributions $\mc{D}$ such that for any estimator $\wh{\mu}$:
    \begin{equation*}
        \max_{D \in \mc{D}} \Pr_{\bm{X} \thicksim D^n} \lbrb{\norm{\wh{\mu} (\bm{X}) - \mu (D)}_{\Sigma (D)} \geq c \sqrt{\frac{d \log (1 / \delta)}{n \log (d)}}} \geq \delta.
    \end{equation*}
\end{theorem}
Next, we present our lower bounds for the adversarial contamination model. Furthermore, our lower bounds hold for the weaker \emph{Huber} contamination model where an adversary is only allowed to \emph{add} corrupted points to the dataset as opposed to corrupting existing points. In the first theorem, we show that the error is unbounded if the corruption fraction exceeds $1 / (d + 1)$ for estimators that are eventually even \emph{approximately} consistent.
\begin{theorem}
    \label{thm:lb_rob_breakdown_afeq}
    For any $d > 3$ and $r > 0$, there exists a family of distributions, $\mc{D}$, and a distribution $D_0 \in \mc{D}$ such that for any $D \in \mc{D}$, there exists distribution $P$ satisfying:
    \begin{equation*}
        D_0 = \frac{d}{d + 1} D + \frac{1}{d + 1} P.
    \end{equation*}
    Furthermore, we have for any estimator, $\wh{\mu}$, and any $n \in \N$:
    \begin{equation*}
        \max_{D \in \mc{D}} \Pr_{\bm{X} \thicksim D_0^n} \lbrb{\norm{\wh{\mu} (\bm{X}) - \mu (D)}_{\Sigma (D)} \geq r} \geq \frac{1}{d + 1}.
    \end{equation*}
\end{theorem}
In the above theorem and the one following it, the distribution $P$ is to be interpreted as the distribution over corruptions that an adversary adds to the sample. Next, we show that the dependence on the corruption fraction when $\eta < 1 / (d+1)$ in \cref{thm:ub_afeq} is tight. 
\begin{theorem}
    \label{thm:lb_rob_quant_afeq}
    For any $d > 3$ and $\eta < 1 / (d + 1)$, there exists a family of distributions, $\mc{D}$, and a distribution $D_0 \in \mc{D}$ such that for any $D \in \mc{D}$, there exists distribution $P$ satisfying:
    \begin{equation*}
        D_0 = (1 - \eta) D + \eta P.
    \end{equation*}
    Furthermore, we have for any estimator, $\wh{\mu}$, and any $n \in \N$:
    \begin{equation*}
        \max_{D \in \mc{D}} \Pr_{\bm{X} \thicksim D_0^n} \lbrb{\norm{\wh{\mu} (\bm{X}) - \mu (D)}_{\Sigma (D)} \geq \frac{1}{2} \sqrt{\frac{d\eta}{1 - d\eta}}} \geq \frac{1}{d + 1}.
    \end{equation*}
\end{theorem}

Taken together, our bounds imply a marked departure from the Euclidean setting. The breakdown point and the dependence of the recovery guarantees on the failure probability and corruption factor \emph{all} decay by a factor of $d$. Intuitively, in the Euclidean setting, the optimal recovery guarantees essentially match what one would achieve when working with \emph{Gaussian} data. However, in the affine-equivariant setting, a significant cost is incurred when weaker assumptions are placed on the data distribution.

Our estimator is based on a novel notion of a high-dimensional median, inspired by the well-known Tukey median \cite{tukeypicturing75} and the Stahel-Donoho estimator \cite{stahel1981,donoho1982breakdown} and may be of independent interest. We aim to find a point whose distance to the mean along any direction is small with respect to the variance along that direction. However, the main difficulty in analyzing the estimator is establishing that such a point always exists. We define an appropriate proxy for the variance which guarantees the existence of such a median while allowing for optimal recovery guarantees. Interestingly, our analysis, similar to that of the Tukey Median, relies strongly on Helly's Theorem, a central result in convex geometry.

The key challenge in proving our lower bounds is establishing the correct dependency on the failure probability in the heavy-tailed setting. Our lower bound construction uses a family of distributions with different covariances in the standard Bayesian estimation-to-testing framework for proving minimax lower bounds (see, for example, \cite{wainwright}). In the typical heavy-tailed setting, where the terms depending on the failure probability and the dimension are decoupled, two lower bounds are established separately for each of them. These are subsequently combined to obtain the final bound. However, in our case, these two elements are intimately coupled making the application of standard techniques challenging. To overcome this, we perform an explicit analysis of the posterior over the set of candidate distributions once the data points have been generated, but only for a carefully chosen set of observations. We show that when such samples are obtained, the posterior is well-spread and that any proposed estimate performs poorly on at least some distributions in the support of the posterior. Here, the differences in the covariance matrices across the distributions in the family play a critical role and the sensitivity of the Mahalanobis norm to such differences yields our lower bound.

\paragraph{Notation:} We use $\N$ and $\R$ to denote the set of natural and real numbers respectively and for $n \in \N$, $[n] \coloneqq \{1, 2, \dots, n\}$. For $d \in \N$, $\R^d$ refers to the standard $d$-dimensional Euclidean vector space with $\bm{1}$, $\bm{0}$, and $\{e_i\}_{i = 1}^d$ denoting the all $1$s vector, the all $0$s vector, and the standard basis vectors respectively. $\mb{S}^{d - 1}$ denotes the unit sphere in $d$ dimensions. For $M_1, M_2 \in R^{d \times d}$, we use $M_1 \succcurlyeq M_2$ (and $M_1 \succ M_2$) indicates that $M_1$ is larger (and strictly larger) than $M_2$ in the standard PSD ordering. For $x \in \R^d$ and $M \in \R^{d \times d}$, $\norm{x}$ and $\norm{M}$ refer to the Euclidean and spectral norm of $x$ and $M$ respectively, and $\Tr M$ denotes the trace of $M$. Furthermore, for $\Sigma \succ 0$, $\norm{x}_\Sigma \coloneqq \sqrt{x^\top \Sigma^{-1} x}$. For a distribution $D$ over $\R^d$, we use $\mu (D)$ and $\Sigma (D)$ to denote its mean and covariance, and for $n \in \N$, $D^n$ denotes the product distribution formed by taking $n$ i.i.d. draws from $D$. We will typically use $\bm{X}$ and $\bm{Y}$ to denote multisets of points in $\R^d$ and with a slight abuse of notation write $\bm{X}, \bm{Y} \subset \R^d$. For a point set $\bm{Y} = \{y_1, \dots, y_n\} \subset \R$, we define:
\begin{equation*}
    \mu(\bm{Y}) \coloneqq \frac{1}{n} \sum_{i = 1}^n y_i,\quad \sigma_1 (\bm{Y}) = \frac{1}{n} \sum_{i = 1}^n \abs{y_i - \mu(\bm{Y})}.
\end{equation*}
Furthermore, $\med (\bm{Y})$ denotes its median and $\abs{\bm{Y}}$ denotes its cardinality. Additionally, for a (possibly infinite) subset $T \subset \R^d$, we use $\mu (T)$ to denote the average of the uniform distribution over $T$. For a proposition $A$, $\bm{1} \lbrb{A}$ is the indicator function which is $1$ if the proposition is true and $0$ otherwise.

In the remainder of the paper, we first overview related work in \cref{sec:rel_work}. We then provide some insight into the failure modes of prior work in \cref{sec:overview}. Next, we formally present our estimator along with a proof of its existence in \cref{sec:alg}. Finally, we establish quantitative upper bounds on its performance in \cref{sec:ub} and present nearly-matching lower bounds on the performance of \emph{any} estimator evaluated in Mahalanobis norm in \cref{sec:lb}.
\section{Related Work}
\label{sec:rel_work}

Here, we review prior work closely related to our own.

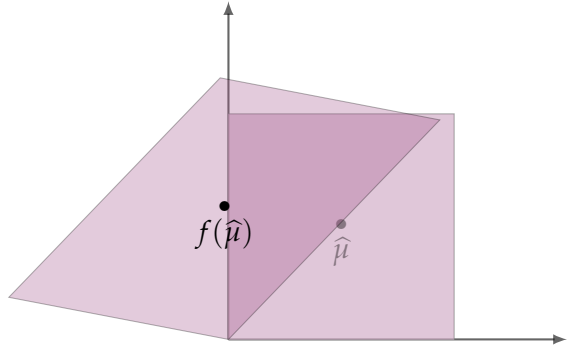
\begin{figure}
    \centering
    \begin{tikzpicture}[scale=1]
        \draw[thick,-latex, opacity=0.6] (0,0) -- (4.5,0);
        \draw[thick,-latex, opacity=0.6] (0,0) -- (0,4.5);

        \filldraw[fill=DarkOrchid, opacity = 0.3, draw=black] (0,0) rectangle (3,3);
        
        \node at (1.5, 1.5) {\textbullet};
        \node at (1.5, 1.5)[anchor = north] {$\wh{\mu}$};

        \filldraw[fill=White, opacity = 0.5, draw=White] (0,0) rectangle (3,3);
        \filldraw[fill=DarkOrchid, opacity = 0.2, draw=black] (0,0) rectangle (3,3);
        
        \filldraw[fill=DarkOrchid, opacity = 0.3,, draw=black] (0,0) -- (2.81, 2.92) -- (-0.11, 3.48) -- (-2.92, 0.56) -- cycle;

        \node at (-0.05, 1.74) {\textbullet};
        \node at (-0.05, 1.74)[anchor = north] {$f(\wh{\mu})$};
    \end{tikzpicture}
    \caption{An illustration of the affine equivariant requirement. Here, the result of an estimator $\wh{\mu}$ run on the transformation of the square to the tilted parallelogram is required to coincide with the transformation of the estimate obtained when run on the square itself.}
    \label{fig:af_eq_illus}
\end{figure}

\paragraph{Affine Equivariant Estimators:} There is a rich line of work in theoretical statistics analyzing the stability and robustness of affine-equivariant estimators. Formally, these estimators commute with affine transformations; i.e for any non-singular affine map $f: \R^d \to \R^d$, an affinely equivariant estimator $\wh{\mu}$ satisfies for any $\bm{X} \subset \R^d$, $\wh{\mu} (f(\bm{X})) = f \lprp{\wh{\mu} (\bm{X})}$ (see \cref{fig:af_eq_illus}). This intuitive property ensures that the results of an inferential procedure utilizing the estimator do not depend on the particular choice of scaling transformations applied to the data. The empirical mean is a prominent example of such an estimator but its performance is severely degraded when faced with heavy-tailed or adversarial data. A significant advance toward the construction of more robust estimators was made by Tukey \cite{tukeypicturing75}, who proposed the Tukey Median as a more robust estimate. The starting point of this estimator is the concept of a \emph{depth function}. Such a function measures how \emph{close to the center} a point is with respect to a set of points. For instance, the depth function corresponding to the Tukey median is defined as follows for $\bm{Y} = \{y_i\}_{i = 1}^n \subset \R$:
\begin{equation*}
    D^1_{\tau} (y; \bm{Y}) = \frac{1}{n}\min \lprp{\abs*{\lbrb{i: y_i \geq y}}, \abs*{\lbrb{i: y_i \leq y}}}.
\end{equation*}
The Tukey Median of a set of points $\bm{X} = \{x_i\}_{i = 1}^n \subset \R^d$ is a point that maximizes its minimum depth over all one-dimensional projections:
\begin{equation*}
    \wh{\mu}_{\tau} (\bm{X}) = \argmax D^d_\tau (x; \bm{X}) \text{ where } D^d_\tau (x; \bm{X}) = \min_{u \in \mb{S}^{d-1}} D^1_{\tau} \lprp{\inp{u}{x}; \lbrb{\inp{u}{x_i}}_{i = 1}^n}.
\end{equation*}
When $d = 1$, this reduces to the standard Median which is known to be robust (and scale-equivariant). The Tukey Median naturally generalizes this simple estimator to higher dimensions. Since then, the breakdown properties of the Tukey Median and other affine-equivariant estimators have been closely investigated by Maronna \cite{mrobustmest76} and Huber \cite{hubrobcov76}. These works concluded that the \emph{breakdown point} \cite{dhnotion83} of these well-known affine-equivariant estimators is at most $1 / (d+1)$. This somewhat disappointing discovery led to the search for estimators with improved breakdown properties. One of the first such approaches was the Stahel-Donoho estimator independently discovered by Stahel \cite{stahel1981} and Donoho \cite{donoho1982breakdown}. Here, one utilizes an alternative notion of \emph{outlyingness}:
\begin{equation*}
    D^1_{\mathrm{SD}} (y; \bm{Y}) = \frac{\abs{y - \med (\bm{Y})}}{\mad (\bm{Y})} \text{ where } \mad (\bm{Y}) = \med \lprp{\lbrb{\abs{y_i - \med (\bm{Y})}}_{i = 1}^n}.
\end{equation*}
The Stahel-Donoho estimate is a point with \emph{minimum} outlyingness:
\begin{equation*}
    \wh{\mu}_{\tau} (\bm{X}) = \argmin D^d_{\mathrm{SD}} (x; \bm{X}) \text{ where } D^d_{\mathrm{SD}} (x; \bm{X}) = \max_{u \in \mb{S}^{d-1}} D^1_{\mathrm{SD}} \lprp{\inp{u}{x}; \lbrb{\inp{u}{x_i}}_{i = 1}^n}
\end{equation*}
This estimator is known to have a breakdown point approaching $1/2$. However, all these approaches have significant drawbacks:
\begin{enumerate}
    \item There exist no quantitative bound on the performance of these estimators.
    \item Furthermore, the attainable bounds depend on the non-degeneracy of the dataset with error bounds growing arbitrarily large as the dataset approaches degeneracy. 
\end{enumerate}
Since the Stahel-Donoho estimator, numerous alternative approaches with differing notions of depth have been proposed: these include estimators based on the simplicial volume \cite{oja83}, S-estimation \cite{rousseeuwyohai84}, the minimum volume ellipsoid \cite{rousseeuw85}, and the simplicial depth \cite{liunotiondepth90}. In addition, the robustness \cite{mrobustmest76,yohai87,davies87,donoholiu88,lrbreakdown91,dsbreakdown92,liu92,lsorderingdirectional92,lps99} and consistency properties \cite{zuoprojdepth03,zchonstahel04} of these estimators have been studied. However, despite this interest, there exist no \emph{quantitative} accuracy guarantees in terms of the number of data points $n$, dimension $d$, failure probability $\delta$, and corruption fraction $\eta$ for these estimators.

\paragraph{Computationally Efficient Adversarially Robust Estimation:} The first \emph{computationally} efficient estimator (in \emph{Euclidean} norm) with near-optimal guarantees was proposed in a breakthrough result of Diakonikolas, Kamath, Kane, Li, Moitra, and Stewart \cite{diakonikolas2016robust}. This estimator is computable in polynomial time with polynomial sample complexity. The statistical and computational complexity has been substantially improved in follow-up works \cite{cdg,dhl} resulting in estimators with near-optimal statistical and computational performance. We direct the interested reader to the excellent survey \cite{diakonikolas2019recent} for more applications of these ideas.

\paragraph{Heavy-tailed Estimation:} An alternative statistical model for outliers is the heavy-tailed corruption model. In this setting, minimal assumptions are made about the data generating distribution (for instance, the covariance of the data-generating distribution exists as opposed to stronger ones such as Gaussianity) and hence, outliers occur naturally as part of the data. This in contrast to an adversary maliciously corrupting the datapoints in the adversarial setting. Here, estimators such as the empirical mean remain \emph{consistent} but suffer from poor \emph{statistical} performance and the emphasis is on designing estimators which avoid this degradation. In one dimension, optimal estimators based on the median-of-means framework have been known (and independently discovered) in a series of classical works \cite{NemYud83,jerrum1986random,alon1999space}. Finding a corresponding \emph{high-dimensional} estimator remained open till the pioneering work of Lugosi and Mendelson \cite{lugosi2017sub} whose estimator achieves the optimal sub-Gaussian rate in the Euclidean setting. A computationally efficient estimator by Hopkins \cite{hopkins2018sub} with the same guarantees followed shortly after along with alternative approaches \cite{lmtrimmed}. Since then, these ideas have been improved and extended to numerous other settings leading to estimators with strong statistical and computational performance \cite{cfb,dllinear,llvz}. Some recent works have also focused on the strong connections between the heavy-tailed and adversarially robust settings yielding estimators simultaneously robust to both corruption models \cite{dllinear,lmtrimmed,hlz20,dkp20}. An alternative line of work has also incorporated privacy guarantees into these estimators \cite{lkko21,hkm22,kmv22,hkmn23}.

In this context, we compare our results to three recently developed estimators: the work of Depersin and Lecue \cite{ldstaheldonoho22}, the setting considered by Duchi, Haque and Kuditipudi \cite{dhk23} and Brown, Hopkins, and Smith \cite{bhs23} which in turn build upon approaches by Brown, Gaboardi, Smith, Ullman, and Zakynthinou \cite{bgsuz}, and the recent result of Lugosi and Mendelson \cite{lmdirectiondependent20}. Depersin and Lecue \cite{ldstaheldonoho22} consider the Stahel-Donoho estimator and show that it achieves sub-Gaussian statistical performance. However, their approach requires strong assumptions on the data-generating process where a multiplicative approximation to $\Sigma$ may essentially be estimated from the data samples. On the other hand, in \cite{bgsuz}, the authors construct affine-equivariant estimators with sub-Gaussian error and strong privacy guarantees with subsequent work \cite{dhk23,bhs23} achieving computational efficiency. However, these works also require the existence of higher order-moments (beyond merely, the existence of the covariance matrix) of the distribution. Finally, sub-Gaussian estimators with direction-dependent accuracy are developed in \cite{lmdirectiondependent20}. However, these bounds scale with the expected \emph{Euclidean} deviation of a sample from its mean which when evaluated in the \emph{Mahalanobis} norm could be arbitrarily large. In addition, this estimator also requires additional assumptions for provable guarantees.
\section{Intuition}
\label{sec:overview}

In this section, we provide some intuition for our estimator. We analyze the performance of two prominent affine-equivariant estimators: the Tukey Median and the Stahel-Donoho estimator. We consider a simple setting where both these estimators perform poorly. We then formally present our estimator and describe how it addresses the shortcomings of these two approaches. We defer the rigorous analysis of our estimator to subsequent sections.

Our hard example will essentially be the simple uniform distribution over the standard basis vectors and the origin $\{e_i\}_{i = 1}^d \cup \{\bm{0}\}$. However, we will assume one of the standard basis vectors (say $e_1$) is mildly more likely to be observed. Formally, the distribution is defined for parameter $\gamma$ as follows:
\begin{equation*}
    \Pr_{X \thicksim D_{\gamma}} \lbrb{X = x} = 
    \begin{cases}
        \frac{1}{d + 1} + \gamma & \text{if } x = e_1 \\
        \frac{1}{d + 1} - \frac{\gamma}{d} & \text{otherwise} 
    \end{cases}.
\end{equation*}
The example is illustrated in $3$ dimensions in \cref{fig:hard_ex}. 
\begin{figure}
    \caption{Illustration of hard distribution. The red dot on $e_1$ denotes higher probability.}
    \label{fig:hard_ex}
    \centering
    \begin{tikzpicture}[scale = 0.8]
        \draw[thick,-latex] (0,0) -- (5,0);
        \draw[thick,-latex] (0,0) -- (0,5);
        \draw[thick,-latex] (0,0) -- (-2.5,-2.5);

        \node[color = red, opacity = 1] at (4, 0) {\small{\textbullet}};
        \node[color = blue, opacity = 1] at (0, 4) {\small{\textbullet}};
        \node[color = blue, opacity = 1] at (0, 0) {\small{\textbullet}};
        \node[color = blue, opacity = 1] at (-1.034, -1.034) {\small{\textbullet}};
    \end{tikzpicture}
\end{figure}
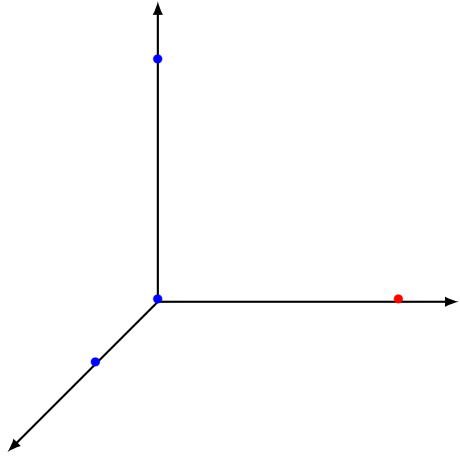

We also assume for the sake of simplicity that the estimators are run directly on the distribution itself with distributional analogues of the corresponding depth and outlyingness functions as opposed to samples from the distribution. We start with the Tukey median and establish that $e_1$ is the unique point with largest Tukey depth. Notice that the depth of $e_1$ is $1 / (d + 1) + \gamma$. Let the support of the distribution be $S$. For any point not in the convex hull of $S$, the separating hyperplane theorem ensures that they have Tukey depth $0$. Now, consider the case where $x$ belongs to the convex hull and $x \neq e_1$. We consider the two possibilities $x \in T$ where $T = S \setminus \{e_1\}$ and $x \notin T$ separately. First, let $x \in T$ and consider the vector $v = x - \bm{1} \lbrb{x = \bm{0}} \bm{1}$. We have $\inp{x}{v} > \inp{y}{v}$ for all $y \in S \setminus x$. Hence, the depth of $x$ is at most $1 / (d + 1) - \frac{\gamma}{d}$. Secondly, consider the alternative case when $x \notin S$ (but lies in its convex hull). We must have:
\begin{equation*}
    x = \sum_{y \in S} w_y y \text{ where } \sum_{y \in S} w_y = 1.
\end{equation*}
Furthermore, since $x \notin S$ and $x \neq e_1$, there exists $y \in T$ with $0 < w_y < 1$. Consider such a vector $y$ and the vector $v = y - \bm{1} \lbrb{y = \bm{0}} \bm{1}$. We now get:
\begin{equation*}
    \forall z \in S \setminus y: \inp{v}{y} > \inp{v}{x} > \inp{v}{z}.
\end{equation*}
Since $y \neq e_1$, the depth of $x$ is also at most $1 / (d + 1) - \frac{\gamma}{d}$. The previous two cases establish that $e_1$ is the unique point of maximum Tukey depth. Unfortunately, the error of $e_1$ is rather large. Consider the one-dimensional projection of the distribution, denoted $D^1_\gamma$, onto $e_1$:
\begin{equation*}
    \Pr_{Y \thicksim D_\gamma^1} \lbrb{Y = y} = 
    \begin{cases}
        \frac{1}{d + 1} + \gamma &\text{if } y = 1 \\
        \frac{d}{d + 1} - \gamma &\text{if } y = 0
    \end{cases}.
\end{equation*}
By considering the error along $e_1$ using \cref{lem:one_d_proj}, we get for $\mu_\gamma = \mu (D_\gamma)$ for $\gamma \leq 1 / (10d)$:
\begin{equation*}
    \norm{e_1 - \mu_\gamma}_{\Sigma (D_\gamma)} \ge \frac{d/(d+1) - \gamma}{\sqrt{\lprp{1 / (d + 1) + \gamma}\lprp{d/(d+1) - \gamma}}} = \sqrt{\frac{d/(d+1) - \gamma}{1 / (d + 1) + \gamma}}  \geq \frac{\sqrt{d}}{2}.
\end{equation*}
As we will see later, this error is larger than optimal by a $\sqrt{d}$ factor. However, notably, the Tukey median can be shown to exist for \emph{any} set of data points. The main drawback of the Tukey median is that it remains insensitive to the \emph{variance} along different directions. As illustrated in \cref{fig:hard_ex_one_d}, the true mean (along $e_1$) lies at $\frac{1}{(d + 1)} + \gamma$ while the Tukey estimate projects to $1$. The variance along $e_1$ is also at most $E[\inp{e_1}{X}^2] = \frac{1}{(d + 1)} + \gamma$. Therefore, incorporating variance information into the estimator can help mitigate some of this degradation.
\begin{figure}
    \centering
    \caption{One dimensional projection onto $e_1$.}
    \label{fig:hard_ex_one_d}
    \begin{tikzpicture}[scale=1]
        \draw[thick,-latex] (0,0) -- (5,0) node [anchor = west] {$e_1$};

        \node[color = blue!50!black] at (0, 0) {\textbullet};
        \node at (0, 0)[anchor = north] {$0$};

        \node[color = red!50!black] at (4, 0) {\textbullet};
        \node at (4, 0)[anchor = north] {$1$};

        \draw[very thick, color = red!50!black] (0.5, 0.1) -- (0.5, -0.1);
        \node at (0.5, 0)[anchor = south] {$\frac{1}{(d + 1)} + \gamma$};
    \end{tikzpicture}
\end{figure}
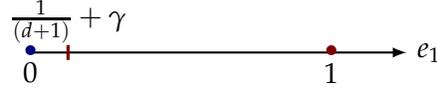

The Stahel-Donoho estimator attempts to incorporate such variance information. However, analyzing the estimator requires non-degeneracy assumptions on the data and even after doing so, these do not provide any \emph{quantifiable} bounds on its performance. For our example, the Stahel-Donoho estimator is not even \emph{defined}. Consider the projection of the distribution onto the standard basis vectors $e_i$ and the all-ones $\bm{1}$ direction.
\begin{figure}
    \centering
    \caption{One dimensional projections onto $e_i$ and $\bm{1}$.}
    \label{fig:hard_ex_two_projs}
    \begin{subfigure}{0.48\textwidth}
    \centering
        \begin{tikzpicture}[scale=1]
            \draw[thick,-latex] (0,0) -- (5,0) node [anchor = west] {$e_i$};
    
            \node[color = blue!50!black] at (0, 0) {\textbullet};
            \node at (0, 0)[anchor = north] {$0$};
            \node at (0, 0)[anchor = south] {$\frac{d}{d + 1} - \gamma$};
    
            \node[color = red!50!black] at (4, 0) {\textbullet};
            \node at (4, 0)[anchor = north] {$1$};
            \node at (4, 0)[anchor = south] {$\frac{1}{d + 1} + \gamma$};
        \end{tikzpicture}
    \end{subfigure}
    \begin{subfigure}{0.48\textwidth}
        \centering
        \begin{tikzpicture}[scale=1]
            \draw[thick,-latex] (0,0) -- (5,0) node [anchor = west] {$\bm{1}$};
    
            \node[color = blue!50!black] at (0, 0) {\textbullet};
            \node at (0, 0)[anchor = north] {$0$};
            \node at (0, 0)[anchor = south] {$\frac{1}{d + 1} - \frac{\gamma}{d}$};
    
            \node[color = red!50!black] at (4, 0) {\textbullet};
            \node at (4, 0)[anchor = north] {$1$};
            \node at (4, 0)[anchor = south] {$\frac{d}{d + 1} + \frac{\gamma}{d}$};
        \end{tikzpicture}
    \end{subfigure}
\end{figure}
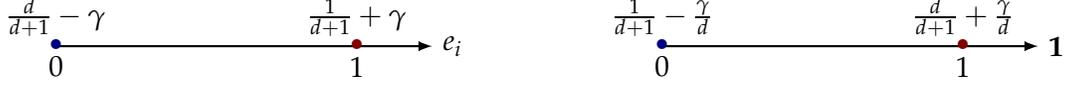
From the one-dimensional projections in \cref{fig:hard_ex_two_projs}, we have the following straightforward observations where $D^1_v$ denotes the projection of $D_\gamma$ onto $v$:
\begin{equation*}
    \med (D^1_{v}) = 
    \begin{cases}
        0 &\text{if } v = e_i \\
        1 &\text{if } v = \bm{1}
    \end{cases}
    \text{ and }
    \forall v \in \{e_i\}_{i = 1}^d \cup \{\bm{1}\}: \mad (D^1_{v}) = 0.
\end{equation*}
Consequently, for an estimate $x$ to have finite Stahel-Donoho outlyingness, it must satisfy $\inp{x}{e_i} = 0$ for all $i$ and $\inp{x}{\bm{1}} = 1$ which is a contradiction. 

Our previous discussion shows that the Tukey median and the Stahel-Donoho estimator fail in two complementary ways. The Tukey median is always defined for any set of data points but its failure to incorporate directional variances into its estimation procedure leads to large error. On the other hand, the variance estimates used in the Stahel-Donoho estimator may not allow for a well-defined estimate in certain settings and even when it is defined, existing analyses do not yield quantitative bounds on its performance. As we will see in \cref{sec:alg}, our estimator simultaneously addresses the shortcomings of both the Tukey median and the Stahel-Donoho estimator. Our median estimator accounts for directional variances like the Stahel-Donoho median but at the same time, is defined for \emph{any} collection of data points like the Tukey median.
\section{Algorithm}
\label{sec:alg}

In this section, we formally present our estimator. We demonstrate how it addresses the shortcomings of the Tukey median and the Stahel-Donoho estimator while simultaneously being well-defined for all point sets and accounting for directional variances. Our estimator is inspired by the Stahel-Donoho estimator but differs in how the robust location and scale parameters are estimated. Recall, the one-dimensional outlyingness function used by the Stahel-Donoho estimator:
\begin{equation*}
    D^1_{\mathrm{SD}} (y; \bm{Y}) = \frac{\abs{y - \med (\bm{Y})}}{\mad (\bm{Y})} \text{ where } \mad (\bm{Y}) = \med \lprp{\lbrb{\abs{y_i - \med (\bm{Y})}}_{i = 1}^n}.
\end{equation*}
The location parameter is robustly estimated by the median and the scale by the mean-absolute deviation ($\mad$) of the one-dimensional point set. The key point of difference is a novel location and scale estimation procedure. Defining for a subset $S \subset [n]$:
\begin{equation*}
    \mu_S (\bm{Y}) \coloneqq \frac{1}{\abs{S}} \sum_{i \in S} y_i \text{ and } \sigma_{1, S} (\bm{Y}) \coloneqq \frac{1}{\abs{S}} \sum_{i \in S} \abs{y_i - \mu_S (\bm{Y})},
\end{equation*}
our location and scale estimates are obtained as follows where $\nu = c / d$ for small constant $c > 0$:
\begin{enumerate}
    \item First, find $S$ satisfying $\abs{S} \geq (1 - \nu) n$ that minimizes $\sigma_{1, S} (Y)$.
    \item Second, define location estimate $\wt{\mu} (\bm{Y}) \coloneqq \mu_S (\bm{Y})$ and scale estimate $\wt{\sigma} (\bm{Y}) \coloneqq \sigma_{1, S} (\bm{Y})$.
\end{enumerate}
With these one-dimensional location and scale estimates, our estimate is defined below:
\begin{gather*}
    D^1_\ours (y; \bm{Y}) = \frac{\abs{y - \wt{\mu} (\bm{Y})}}{\wt{\sigma} (\bm{Y})} \\
    \wh{\mu}_{\ours} (\bm{X}) = \argmin D^d_{\ours} (x; \bm{X}) \text{ where } D^d_\ours (x; \bm{X}) \coloneqq \max_{v \in \mb{S}^{d-1}} D^1_\ours \lprp{\inp{x}{v}; \{\inp{x_i}{v}\}_{i = 1}^n}
\end{gather*}
We show that with these definitions, our estimate \emph{always} exists for \emph{any} point set with finite outlyingness. In fact, we establish the following \emph{strengthening} of this statement:
\begin{enumerate}
    \item Firstly, we show that the estimate always has \emph{constant} outlyingness. This allows us to prove sharp quantitative bounds in our settings of interest.
    \item Secondly, while our definition technically requires choosing $S$ to minimize the directional scale estimate, we show that there exists an estimate with finite outlyingness for \emph{all} choices of $S$ satisfying the size constraints for every direction $v$.
\end{enumerate}

This median is defined in \cref{alg:gentuk} and the proof of its existence is provided in \cref{thm:gen_tuk_existence}. The proof relies on Helly's Theorem (\cref{thm:helly}), a fundamental result in convex geometry.

\begin{algorithm}[H]
\caption{High-dimensional Median}
\label{alg:gentuk}
\begin{algorithmic}[1]
    \STATE \textbf{Input}: Point set $\bm{X} = \{x_i\}_{i = 1}^k \subset \R^d$
    \STATE Let $\nu = 1 / (3  d)$ and $\mc{S} = \{S \subset [k]: \abs{S} \geq (1 - \nu) k\}$
    \STATE Define for all $v\in \mb{S}^{d - 1}, S \in \mc{S}$:
    \begin{gather*}
        \mu_{v, S} = \mu_S \lprp{\lbrb{\inp{v}{x_i}}_{i = 1}^k} \qquad \sigma_{v, S} = \sigma_{1, S} \lprp{\lbrb{\inp{v}{x_i}}_{i = 1}^k}
    \end{gather*}
    \STATE Define convex compact sets:
    \begin{equation*}
        T_{v, S} = \lbrb{x \in \R^d: \abs{\inp{x}{v} - \mu_{v,S}} \leq 2 \sigma_{v,s}} \cap \conv(\bm{X})
    \end{equation*}
    \STATE Let $T = \cap_{v \in \mb{S}^{d - 1}, S \in \mc{S}} T_{v, S}$
    \STATE \textbf{Return: } $\mu (T)$
\end{algorithmic}
\end{algorithm}

For a point set $\bm{X} = \{x_i\}_{i = 1}^n \subset \R^d$, let $\wh{\mu} (\bm{X})$ denote the output of \cref{alg:gentuk}. The main result of this section establishes the existence of this estimator. 
\begin{theorem}
    \label{thm:gen_tuk_existence}
    For any $k \in \N$ and $\bm{X} = \{x_i\}_{i = 1}^k \subset \R^d$, $\wh{\mu} (\bm{X})$ is well defined and affine-equivariant.
\end{theorem}
\begin{proof}
    We tackle the two claims of the theorem in turn.

    \paragraph{Existence of $\wh{\mu}$:} We first show that $T$ is non-empty, convex, and compact implying the first claim. Note that $T$ is the intersection of compact convex sets and is hence, convex and compact. To establish the non-emptiness of $T$, an application of Helly's Theorem (\cref{thm:helly}) allows us to restrict to intersections of at most $d + 1$ of the sets $T_v$. Consider any $d + 1$ sized collection $H = \{v_j, S_j\}_{j \in [d + 1]}$. We have:
    \begin{equation*}
        R := \cap_{j \in [d + 1]} S_{j}, \abs{R} \geq (1 - (d + 1) \nu)k \geq \frac{k}{2}.
    \end{equation*}
    Defining:
    \begin{equation*}
        \mu_R = \frac{1}{\abs{R}} \cdot \sum_{i \in R} x_i,
    \end{equation*}
    we will show that $\mu_R$ lies in $\cap_{(v, S) \in H} T_{v, S}$. For any $(v, S) \in H$, we have:
    \begin{align*}
        \abs{\inp{\mu_R}{v} - \mu_{v, S}} &= \abs*{\frac{1}{\abs{R}} \sum_{i \in R} (\inp{x_i}{v} - \mu_{v, S})} \leq \frac{1}{\abs{R}} \cdot \sum_{i \in R} \abs*{\inp{x_i}{v} - \mu_{v, S}} \leq \frac{\abs{S}}{\abs{R}} \sigma_{v, S}\leq 2\sigma_{v, S}.
    \end{align*}
    An application of Helly's theorem now establishes that $T$ is non-empty proving the theorem.

    \paragraph{Affine-equivariance of $\wh{\mu}$:} Next, we establish that $\wh{\mu}$ is affine-equivariant. Let $\bm{X} = \{x_i\}_{i = 1}^k \subset \R^{d}$ and $f(x) = Ax + b$ with $A \in \R^{d \times d}, b \in \R^d$ and $A$ be non-singular. Hence, $f$ is an invertible affine transformation. Furthermore, let $\bm{X}^\prime = f(\bm{X}) = \{x_i^\prime = f(x_i)\}_{i = 1}^k$ and $T$ be the set obtained in \cref{alg:gentuk} on input $\bm{X}$ and $T^\prime$ be the corresponding set on $f(\bm{X}^\prime)$. We will show $T^\prime = f(T)$ proving the second claim. 
    
    First, let $x \in T$ and we prove $f(x) \in T^\prime$. Observe for any $v \in \mb{S}^{d - 1}, i \in [k]$:
    \begin{align*}
        \inp{v}{x^\prime_i} = \inp{v}{f(x_i)} &= \inp{v}{A x_i} + \inp{v}{b} = \inp{v}{A x_i} + \inp{v}{b} \\
        &= \inp{A^\top v}{x_i} + \inp{v}{b} = \norm{A^\top v}\inp*{\frac{A^\top v}{\norm{A^\top v}}}{x_i} + \inp{v}{b}.
    \end{align*}
    We have by defining $v^\prime = \frac{A^\top v}{\norm{A^\top v}}$ for any $S \subset [k]$ with $\abs{S} \geq (1 - \nu)k$:
    \begin{gather*}
        \mu \lprp{\{\inp{v}{x^\prime_i}\}_{i \in S}} = \norm{A^\top v} \cdot \mu \lprp{\{\inp{v^\prime}{x_i}\}_{i \in S}} + \inp{v}{b} \\
        \sigma_1 \lprp{\{\inp{v}{x^\prime_i}\}_{i \in S}} = \norm{A^\top v} \cdot \sigma_1 \lprp{\{\inp{v^\prime}{x_i}\}_{i \in S}}.
    \end{gather*}
    As a consequence, we get:
    \begin{align*}
        \abs{\inp{v}{Ax + b} - \mu \lprp{\{\inp{v}{x^\prime_i}\}_{i \in S}}} &= \norm{A^\top v} \cdot \abs*{\inp{v^\prime}{x} - \mu \lprp{\{\inp{v^\prime}{x_i}\}_{i \in S}}} \\
        &\leq 2 \cdot \norm{A^\top v} \cdot \sigma_1 \lprp{\{\inp{v^\prime}{x_i}\}_{i \in S}} = 2\sigma_1 \lprp{\{\inp{v}{x^\prime_i}\}_{i \in S}}
    \end{align*}
    where the inequality follows from $x \in T$. Since, the above holds for all $x \in T, v \in \mb{S}^{d - 1}, S \subset [k]$ with $\abs{S} \geq (1 - \nu)k$, we get that $f(T) \subseteq T^\prime$. By repeating the above argument for $f^{-1} (z) = A^{-1} z - A^{-1} b$, we get that $f^{-1} (T^\prime) \subset T$ which implies $T^\prime \subset f(T)$ concluding the proof.
\end{proof}
\section{Theoretical Guarantees} 
\label{sec:ub}

Here, we prove \cref{thm:ub_afeq} using the high-dimensional median described in the previous section. Our estimator achieving the guarantees of \cref{thm:ub_afeq} is defined in \cref{alg:est_afeq}. Note that since our high-dimensional median is affine-equivariant (\cref{thm:gen_tuk_existence}), so is \cref{alg:est_afeq}. Hence, it suffices to establish \cref{thm:ub_afeq} in the setting $\mu = 0$ and $\Sigma = I$.
\begin{algorithm}[H]
    \caption{Affine-equivariant Estimator}
    \label{alg:est_afeq}
    \begin{algorithmic}[1]
        \STATE \textbf{Input}: Point set $\bm{X} = \{x_i\}_{i = 1}^n \subset \R^d$, Confidence Parameter $\delta$
        \STATE $k \gets \max (6 \eta d n, C d \log (1 / \delta))$
        \STATE Partition $\bm{X}$ into $k$ equally sized buckets $\{\mc{B}_i\}_{i \in [k]}$
        \STATE Compute $\wh{\mu}_i = \mu (\mc{B}_i)$
        \STATE $\wh{\mu} = \text{High-dimensional Median} (\{\wh{\mu}_i\}_{i \in [k]})$
        \STATE \textbf{Return:} $\wh{\mu}$
    \end{algorithmic}
\end{algorithm}

We first prove the following technical lemma which establishes the required concentration properties on the bucketed means, $\wh{\mu}_i$'s. Before we proceed, we define the thresholding operator for a threshold $\tau \geq 0$ as follows where $\sgn(\cdot)$ denotes the signum function:
\begin{equation*}
    \psi_\tau (x) = 
   \begin{cases}
       x & \text{if } \abs{x} \leq \tau \\
       \sgn(x) \tau & \text{otherwise}
   \end{cases}.
\end{equation*}
\begin{lemma}
    \label{lem:conc_afeq}
    There exists an absolute constant $C > 0$ such that the following holds. Let $Y_1, \dots, Y_k$ be $k$ i.i.d. random vectors drawn from a distribution $D$ with mean $0$ and variance $\sigma^2 I$. Then, we have for $\tau = 24\sigma d$:
    \begin{equation*}
        \max_{v \in \mb{S}^{d-1}} \frac{1}{k} \sum_{i = 1}^k \abs{\psi_\tau (\inp{v}{Y_i})} \leq 2\sigma
    \end{equation*}
    with probability at least $1 - \delta$ when $k \geq C d \log (1 / \delta)$. 
\end{lemma}
\begin{proof}
    Define the random variable $Z$ as follows with $Y$ drawn from $D$:
    \begin{equation*}
        Z = \max_{v \in \mb{S}^{d - 1}} \frac{1}{k} \sum_{i = 1}^k \abs{\psi_\tau\lprp{\inp{Y_i}{v}}} - \E \lsrs{\abs{\psi_\tau\lprp{\inp{Y}{v}}}}.
    \end{equation*}
    We first bound the expectation of $Z$ with $Y_i^\prime$ and $\gamma_i$ denoting i.i.d. vectors from $D$ and Rademacher random variables respectively. The fourth inequality follows from the Ledoux-Talagrand contraction principle (\cref{cor:ledtal}) and the observation that $\abs{\psi_\tau (\cdot)}$ is $1$-Lipschitz:
    \begin{align*}
        \E [Z] &\leq \E \lsrs{\max_{v \in \mb{S}^{d - 1}} \abs*{\frac{1}{k} \sum_{i = 1}^k \abs{\psi_\tau\lprp{\inp{Y_i}{v}}} - \E \lsrs{\abs{\psi_\tau\lprp{\inp{Y}{v}}}}}} \\
        &\leq \E \lsrs{\max_{v \in \mb{S}^{d - 1}} \abs*{\frac{1}{k} \sum_{i = 1}^k \abs{\psi_\tau\lprp{\inp{Y_i}{v}}} - \abs{\psi_\tau\lprp{\inp{Y_i^\prime}{v}}}}} \\
        &= \E \lsrs{\max_{v \in \mb{S}^{d - 1}} \abs*{\frac{1}{k} \sum_{i = 1}^k \gamma_i \lprp{\abs{\psi_\tau\lprp{\inp{Y_i}{v}}} - \abs{\psi_\tau\lprp{\inp{Y_i^\prime}{v}}}}}} \\
        &\leq 2 \E \lsrs{\max_{v \in \mb{S}^{d - 1}} \abs*{\frac{1}{k} \sum_{i = 1}^k \gamma_i \abs{\psi_\tau\lprp{\inp{Y_i}{v}}}}} \leq 4 \E \lsrs{\max_{v \in \mb{S}^{d - 1}} \abs*{\frac{1}{k} \sum_{i = 1}^k \gamma_i \inp{Y_i}{v}}} \\
        &= \frac{4}{k} \E \lsrs{\norm*{\sum_{i = 1}^{k} \gamma_i Y_i}} \leq \frac{4}{k} \sqrt{\E \lsrs{\norm*{\sum_{i = 1}^{k} \gamma_i Y_i}^2}} = 4 \sqrt{\frac{d}{k}} \sigma.
    \end{align*}
    Additionally, noting that $\psi_\tau (x) \leq \tau$ for all $x \in \R$:
    \begin{equation*}
        Y_{i, v} \coloneqq \frac{1}{\tau} \lprp{\abs{\psi_\tau (\inp{Y_i}{v})} - \E [\abs{\psi_\tau (\inp{Y}{v})}]} \leq 1.
    \end{equation*}
    Furthermore, we have for all $v \in \mb{S}^{d - 1}$:
    \begin{equation*}
        \sum_{i = 1}^k \E [Y_{i, v}^2] \leq \frac{k}{\tau^2} \E \lsrs{\psi_\tau\lprp{\inp{Y}{v}}^2} \leq \frac{k}{\tau^2} \E \lsrs{\inp{Y}{v}^2} = \frac{k \sigma^2}{\tau^2}.
    \end{equation*}
    Hence, we get by an application of Bousquet's inequality (\cref{thm:bousequet_thm}):
    \begin{equation*}
        \Pr \lbrb{Z \geq \E [Z] + t} \leq \exp \lprp{- \lprp{\frac{k}{\tau}}^2 \cdot \frac{t^2}{2(v + kt/(3\tau))}} \text{ where } v = \frac{8\sigma\sqrt{kd}}{\tau} + \frac{k\sigma^2}{\tau^2}.
    \end{equation*}
    Setting $t = \frac{\sigma}{2}$ and from our setting of $\tau$ and $k$, we get:
    \begin{equation*}
        Z \leq \sigma
    \end{equation*}
    with probability at least $1 - \delta$. The lemma now follows as:
    \begin{equation*}
        \E \lsrs{\abs{\psi_\tau\lprp{\inp{Y}{v}}}} \le \sqrt{\Ex \lsrs{\psi_\tau\lprp{\inp{Y}{v}}^2}} \le \sqrt{\Ex \lsrs{\inp{Y}{v}^2}} \le \sigma.
    \end{equation*}
\end{proof}

We now proceed to the proof of \cref{thm:ub_afeq}. For the sake of analysis let $\wt{\mu}_i$ denote the \emph{uncorrupted} versions of the bucketed means $\wh{\mu}_i$. For these, we have:
\begin{equation*}
    \E [\wt{\mu}_i] = 0 \qquad \text{ and } \qquad \E [\wt{\mu}_i \wt{\mu}_i^\top] = \frac{k}{n} I.
\end{equation*}
Hence, we get by \cref{lem:conc_afeq} and the setting of $k$ in \cref{alg:est_afeq}:
\begin{equation*}
    \forall \norm{v} = 1: \frac{1}{k} \sum_{i = 1}^k \abs{\psi_\tau (\inp{v}{\wt{\mu}_i})} \leq 2\wt{\sigma} \text{ where } \wt{\sigma} = \sqrt{\frac{k}{n}} \text{ and } \tau = 24 \wt{\sigma}d
\end{equation*}
with probability at least $1 - \delta$. We condition on this event in the remainder of the proof. Note, furthermore, that there are at most $\eta n$ many corrupted points in $\bm{X}$. Therefore, we have for $\abs{\{i: \wt{\mu}_i = \wh{\mu}_i\}} \geq (1 - 1 / (6d)) k$ again from the setting of $k$ in \cref{alg:est_afeq} and that:
\begin{equation*}
    \forall \norm{v} = 1: \frac{1}{k} \sum_{i = 1}^k \bm{1} \lbrb{\abs{\inp{v}{\wt{\mu}_i}} \geq \tau} \leq \frac{1}{\tau} \cdot \frac{1}{k} \sum_{i = 1}^k \abs{\psi_\tau (\inp{v}{\wt{\mu}_i})} \leq \frac{1}{12d}.
\end{equation*}
Therefore, we get from the previous two observations that:
\begin{equation*}
    \forall \norm{v} = 1: \abs{\mc{G}_v} \geq \lprp{1 - \frac{1}{4d}} k \text{ where } \mc{G}_v = \lbrb{i: \wt{\mu}_i = \wh{\mu}_i \text{ and } \psi_\tau (\inp{v}{\wt{\mu}_i}) = \inp{v}{\wt{\mu}_i}}.
\end{equation*}
Now, let $v \in \mb{S}^{d - 1}$. We have for $\mc{G}_v$ from \cref{alg:gentuk,thm:gen_tuk_existence}:
\begin{equation*}
    \abs{\inp{v}{\wh{\mu}} - \mu (\{\inp{v}{\wh{\mu}_i}\}_{i \in \mc{G}_v})} \leq 2 \sigma_1 (\{\inp{v}{\wh{\mu}_i}\}_{i \in \mc{G}_v}).
\end{equation*}
For the mean term, we get:
\begin{align*}
    \abs{\mu (\{\inp{v}{\wh{\mu}_i}\}_{i \in \mc{G}_v})} & \leq \mu (\{\abs{\inp{v}{\wh{\mu}_i}}\}_{i \in \mc{G}_v}) \leq \frac{1}{(1 - 1 / (4d))} \mu (\{\abs{\psi_\tau(\inp{v}{\wt{\mu}_i})}\}_{i \in [k]}) \leq 3 \wt{\sigma}.
\end{align*}

For the deviation term, we get:
\begin{align*}
    \sigma_1 (\{\inp{v}{\wh{\mu}_i}\}_{i \in \mc{G}_v}) &= \mu (\{\abs{\inp{v}{\wh{\mu}_i} - \mu(\{\inp{v}{\wh{\mu}_i}\}_{i \in \mc{G}_v}) }\}_{i \in \mc{G}_v}) \leq \mu (\{\abs{\inp{v}{\wh{\mu}_i}}\}_{i \in \mc{G}_v}) + 3\wt{\sigma} \leq 6\wt{\sigma}.
\end{align*}
The above two bounds imply:
\begin{equation*}
    \forall v \in \mb{S}^{d - 1}: \abs{\inp{v}{\wh{\mu}}} \leq 15\wt{\sigma} = 15 \sqrt{\frac{k}{n}}
\end{equation*}
establishing the theorem. \qed

We conclude the section with the following remarks:

For clarity and ease of exposition, we impose a minor constraint that $\eta \in [0, 1/(6d)]$. The analysis reveals that the results are applicable for any $\eta \le c/(d+1)$ with any constant $c < 1$. Moreover, in the absence of adversarial contamination or bucketing, our analysis yields an $O(1)$ upper bound in the distributional setting. This stands in contrast with the $\Omega(\sqrt{d})$ lower bound of Tukey median, as elucidated in \cref{sec:overview}.
\section{Lower Bounds}
\label{sec:lb}

Here, we present the proofs of \cref{thm:lb_ht_afeq,thm:lb_rob_breakdown_afeq,thm:lb_rob_quant_afeq} which show that the guarantees of \cref{thm:ub_afeq} are nearly tight. For the heavy-tailed setting with \emph{no} adversarial contamination (i.e $\eta = 0$), \cref{thm:lb_ht_afeq,thm:lb_rob_breakdown_afeq,thm:lb_rob_quant_afeq} shows that the recovery error of our estimator is optimal up to a $\sqrt{\log (d)}$ factor and for the adversarial contamination model, \cref{thm:lb_rob_breakdown_afeq,thm:lb_rob_quant_afeq} establishes that \emph{no} affinely-equivariant estimator can achieve breakdown point greater than $1 / (d + 1)$ and that $O(\sqrt{d\eta})$ is the best achievable recovery error for \emph{any} affine-equivariant estimator.

\subsection{Heavy-tailed Lower Bound - Proof of \cref{thm:lb_ht_afeq}}

To define our class of distributions, let:
\begin{equation*}
    \eps = \frac{1}{4}\sqrt{\frac{d \log (1 / (d\delta))}{n \log (d)}}.
\end{equation*}
Our hard class will contain $d$ distributions with support over the standard basis vectors and the origin, i.e., $\{e_i\}_{i = 1}^d \cup \lbrb{\bm{0}}$ such that each distribution puts a smaller mass at one of the standard basis vectors. More formally, we have $\mc{D} = \{D_i\}_{i = 1}^{d}$ with:
\begin{equation*}
    \Pr_{X \thicksim D_i} \lbrb{X = e_j} = 
    \begin{cases}
        \frac{\eps^2}{d} & \text{if } i \neq j \\
        \frac{\eps^2}{d^2} & \text{if } i = j \\
    \end{cases},       
\end{equation*}
and 
\begin{equation*}
    \Pr_{X \sim D_i} \lbrb{X = \bm{0}} = 1 - \frac{d-1}{d}\eps^2 - \frac{\eps^2}{d^2}.
\end{equation*}
By a straightforward calculation, we have:
\begin{equation*}
    \Sigma (D_i) \preccurlyeq M^i \text{ where } M^i_{j k} = 
    \begin{cases}
        0 & \text{if } j \neq k \\
        \frac{\eps^2}{d} & \text{if } j = k \text{ and } j \neq i \\
        \frac{\eps^2}{d^2} & \text{if $j = k = i$}
    \end{cases}.
\end{equation*}

Now consider the following procedure of generating the data $X$:
\begin{enumerate}
    \item Sample a random integer $I$ from the index set $\lbrb{1,2,\dots,d}$.
    \item Given $I = i$, draw $n$ i.i.d samples $\bm{X} = \lbrb{X_1, X_2, \dots, X_n}$ from $D_i$.
\end{enumerate}

Next, it suffices to show that for any estimator $\wh{\mu}(\bm{X})$, we have 
\begin{equation*}
    \Pr \lbrb{\norm{\wh{\mu} (\bm{X}) - \mu (D_I)}_{\Sigma(D_I)} \geq  \frac{1}{4}\eps} \geq \delta.
\end{equation*}

For each distribution $D_i$, consider a set of instances
\begin{gather*}
    S_i = \lbrb{\bm{X} = (x_1, \dots, x_n): m_i (\bm{X}) \geq \frac{4 \eps^2 n}{d} \text{ and } \sum_{j = 1}^d \bm{1} \lbrb{m_j (\bm{X}) < \frac{4 \eps^2 n}{d}} \geq \frac{d}{2}} \\
    \text{ where } m_j (\bm{X}) := \sum_{k = 1}^n \bm{1} \lbrb{x_k = e_j}.
\end{gather*}

Next, consider $S:=\cup S_i$. Now for any $\bm{X} \in S$, let
\begin{equation*}
    \mc{J} := \lbrb{j: m_j (\bm{X}) < \frac{4\eps^2n}{d}} \text{ and } z := \wh{\mu} (\bm{X}).
\end{equation*}

Suppose $z_j \leq \frac{\eps^2}{2d}$ for all $j \in \mc{J}$. Then, by considering the cumulative error on $\mc{J}$ we have
\begin{equation*}
    \norm{\wh{\mu} (\bm{X}) - \mu (D_k)}^2_{\Sigma(D_k)} \ge \lprp{\frac{d}{2} - 1} \frac{\lprp{\eps^2 / 2d}^2}{\eps^2 / d} \ge \frac{1}{16} \eps^2
\end{equation*}
where we use \cref{lem:psd_maha} with $M^k$ and $\Sigma (D_k)$ for any $D_k \in \mc{D}$.

On the other hand, suppose there exists $j \in \mc{J}$ such that $z_j > \frac{\eps^2}{2d}$. If $I = j$ is the sampled index, then by considering the error on $e_j$ we have by \cref{lem:one_d_proj}
\begin{equation*}
    \norm{\wh{\mu} (\bm{X}) - \mu (D_j)}_{\Sigma(D_j)} \ge \frac{\abs{\eps^2/2d - \eps^2/d^2}}{\sqrt{\eps^2/d^2}} \ge \frac{1}{4} \eps.
\end{equation*}

By the definition of $\bm{X}$, let $i$ be the index such that $m_i(\bm{X}) \ge \frac{4 \eps^2 n}{d}$, then we have the posterior probability of $I = j$ is at least that of $I = i$. From the previous two inequalities, we obtain
\begin{equation*}
    \Pr \lbrb{\norm{\wh{\mu} (\bm{X}) - \mu (D_I)}_{\Sigma(D_I)} \ge \frac{1}{4} \eps \bigg| \bm{X}} \geq \Pr \lbrb{I = j | \bm{X}} \geq \Pr \lbrb{I = i | \bm{X}}.
\end{equation*}

From the above, we have
\begin{align*}
    \Pr \lbrb{\norm{\wh{\mu} (\bm{X}) - \mu (D_I)}_{\Sigma(D_I)} \ge \frac{1}{4} \eps} &\geq \sum_{\bm{X} \in S_i} \Pr \lbrb{\norm{\wh{\mu} (\bm{X}) - \mu (D_I)}_{\Sigma(D_I)} \ge \frac{1}{4} \eps \bigg| \bm{X}} \Pr \lbrb{\bm{X}} \\
    &\geq \sum_{\bm{X} \in S_i} \Pr \lbrb{I = i | \bm{X}} \Pr \lbrb{\bm{X}} = \sum_{\bm{X} \in S_i} \Pr \lbrb{I = i, \bm{X}} \\
    &= \sum_{\bm{X} \in S_i} \Pr \lbrb{\bm{X} | I = i} \Pr \lbrb{I = i}.
\end{align*}

It remains to prove that  $\Pr \lbrb{S_i | I = i} \ge d \delta$. For simplicity of notation, denote $\Pr_i$ as the conditional distribution of $X$ under $I = i$. Define events:
\begin{gather*}
    A = \lbrb{\bm{X} = (x_1, \dots, x_n): m_i (\bm{X}) \geq \frac{4 \eps^2 n}{d}} \\
    B = \lbrb{\bm{X} = (x_1, \dots, x_n): \sum_{j = 1, j \ne i}^d \bm{1} \lbrb{m_j (\bm{X}) < \frac{4 \eps^2 n}{d}} \geq \frac{d}{2}} \\
    C = \lbrb{\bm{X} = (x_1, \dots, x_n): m_0 (\bm{X}) \ge (1 - 2\eps^2) n} \text{ where } m_0(\bm{X}) := \sum_{k = 1}^n \bm{1} \lbrb{x_k = 0}.
\end{gather*}

Note that $\Pr_i\lbrb{S_i} = \Pr_i \lbrb{A \cap B}$ and $C \subseteq B$. So we have 
\begin{equation*}
    \P_i (A \cap B) = \P_i (B|A) \P_i(A) \ge \P_i (B) \P_i(A) \ge \P_i(C) \P_i(A).
\end{equation*}
where the first inequality follows from \cref{lem:mult_cond_counts}.

We first use a Binomial tail lower bound to bound $\Pr_i(A)$ (see e.g., \cite{ash2012information}):

    \begin{equation*}
        \P_i\lbrb{ B(n, p) \ge k} \ge \frac{1}{\sqrt{8n\frac{k}{n}\lprp{1 - \frac{k}{n}}}} \exp\lprp{-n D \infdivx*{\frac{k}{n}}{p}},
    \end{equation*}
    where $B(n,p)$ denotes a Binomial random variable and $D\infdivx{a}{p}= a \log\frac{a}{p} + (1-a)\log\frac{1-a}{1-p}$ denotes the KL divergence.

    Plugging in $k = 4dnp$ and $p = \eps^2/d^2$ we obtain that 
\begin{equation*}
    \P_i(A) \ge \frac{1}{\sqrt{32ndp}} \exp\lprp{-4ndp\log\lprp{4d}} \ge 2d\delta.
\end{equation*}

Also, note that $\Pr(C) \ge 1/2$ since $m_0(x)$ follows a Binomial distribution with $\E [m_0(x)] \ge (1-\eps^2)n$. Therefore, we have $\Pr_i(S_i) \ge d\delta$ concluding the proof. \qed

\subsection{Adversarial Contamination - Proofs of \cref{thm:lb_rob_breakdown_afeq,thm:lb_rob_quant_afeq}}

We start with \cref{thm:lb_rob_breakdown_afeq}.

\paragraph{Proof of \cref{thm:lb_rob_breakdown_afeq}.} Let $r > 0$ and $S = \{e_i\}_{i = 1}^d \cup \{\bm{0}\}$. First define the family $\wt{\mc{D}} = \{\wt{D}_i\}_{i = 0}^{d + 1}$ with $\wt{D}_0$ and $\wt{D}_{d + 1}$ denoting the uniform distributions over $S$ and $S \setminus \{\bm{0}\}$ respectively and for $i \in [d]$, $\wt{D}_i$ is defined as follows:
\begin{equation*}
    \Pr_{X \thicksim \wt{D}_i} \lbrb{X = x} = 
    \begin{cases}
        0 & \text{if } x = e_i \\
        \frac{1}{d} & \text{if } x \in S \setminus \{e_i\}
    \end{cases}.
\end{equation*}
Defining $\sigma = 1 / (2dr)$, our hard family of distributions $\mc{D} = \{D_i\}_{i = 0}^{d + 1}$ is defined in the following way:
\begin{enumerate}
    \item First, generate $\wt{X} \thicksim \wt{D}_i$
    \item Independently, generate $Z \thicksim \mrm{Unif} (\{\pm 1\}^d)$
    \item Observe $X = \wt{X} + \sigma Z$.
\end{enumerate}
Note that $\Sigma (D_i)$ is non-singular for each $i$ and $\wt{D}_0$ satisfies:
\begin{equation*}
    \wt{D}_0 = \frac{d}{d + 1} \wt{D}_i + \frac{1}{d + 1} \delta_{e_i}
\end{equation*}
where $\delta_{e_i}$ denotes a Dirac-delta distribution on $e_i$. Suppose $\wh{\mu}$ satisfies for some $n \in \N$:
\begin{equation*}
    \forall D \in \mc{D}: \Pr_{\bm{X} \thicksim D_0^n} \lbrb{\norm{\wh{\mu} (\bm{X}) - \mu (D)}_{\Sigma (D)} \geq r} < \frac{1}{d + 1}.
\end{equation*}
Then, by the union bound, there must exist a sample $\bm{X}$ in the support of $D_0^n$ and an outcome $\wh{\mu} = \wh{\mu} (\bm{X})$ over the randomness of $\wh{\mu}$ such that:
\begin{equation*}
    \forall D \in \mc{D} \setminus \{D_0\}: \norm{\wh{\mu} (\bm{X}) - \mu (D)}_{\Sigma (D)} \leq r.
\end{equation*}

Then, letting $\wh{\mu} = \wh{\mu} (\bm{X})$ and considering the direction $e_i$ for any $i \in [d]$, we have by \cref{lem:one_d_proj}:
\begin{equation*}
    r \geq \norm{\wh{\mu} - \mu (D_i)}_{\Sigma (D_i)} \geq 2dr \abs{\wh{\mu}_i}.
\end{equation*}
This, implies:
\begin{equation*}
    \abs{\wh{\mu}_i} \leq \frac{1}{2d} \implies \sum_{i = 1}^d \abs{\wh{\mu}_i} \leq \frac{1}{2}.
\end{equation*}
However, note that we have for the direction $\bm{1}$ and the distribution $D_{d + 1}$ again by \cref{lem:one_d_proj}:
\begin{equation*}
    r \geq \norm{\wh{\mu} - \mu (D_{d + 1})}_{\Sigma (D_{d + 1})} \geq \abs*{\frac{1 - \sum_{i = 1}^d \wh{\mu}_i}{\sqrt{d}/(2dr)}} \geq \sqrt{d} r,
\end{equation*}
which is a contradiction thus establishing the theorem. \qed

We now move on to \cref{thm:lb_rob_quant_afeq}.

\paragraph{Proof of \cref{thm:lb_rob_quant_afeq}.} As before, we will construct a hard family of distributions. For support set $S = \{e_i\}_{i = 1}^d \cup \{\bm{1} / d\}$, define the set of distributions $\wt{\mc{D}} = \{\wt{D}_i\}_{i = 0}^{d}$ as follows:
\begin{gather*}
    \forall i \in [d]: \Pr_{X \thicksim \wt{D}_i} (X = x) =
    \begin{cases}
        0 &\text{if } x = e_i \\
        \frac{d}{d - 1} \eta &\text{if } x = e_j \text{ for } j \neq i \\
        1 - d\eta &\text{if } x = \frac{\bm{1}}{d}
    \end{cases}
    \text{ and } \\
    \Pr_{X \thicksim \wt{D}_0} (X = x) =
    \begin{cases}
        \eta &\text{if } x = e_j \text{ for any } j \in [d]\\
        1 - d\eta &\text{if } x = \frac{\bm{1}}{d}
    \end{cases}.
\end{gather*}
Let
\begin{equation*}
    r \coloneqq \frac{1}{2} \sqrt{\frac{d\eta}{1 - d\eta}} \text{ and } \sigma \coloneqq \frac{\eta}{4r}.
\end{equation*}
The hard family of distributions, $\mc{D} = \{D_i\}_{i = 0}^d$ is defined as follows:
\begin{enumerate}
    \item First, generate $\wt{X} \thicksim \wt{D}_i$
    \item Independently, generate $Z \thicksim \mrm{Unif} (\{\pm 1\}^d)$
    \item Observe $X = \wt{X} + \sigma Z$.
\end{enumerate}
As before, for each $i$, $\Sigma (D_i)$ is non-singular and $D_0$ may be written as a mixture of $D_i$ and the distribution with all its mass on $e_i$. Now, suppose $\wh{\mu}$ is an estimator that satisfies for some $n \in \N$:
\begin{equation*}
    \forall D \in \mc{D}: \Pr_{\bm{X} \thicksim D_0^n} \lbrb{\norm{\wh{\mu} (\bm{X}) - \mu (D)}_{\Sigma (D)} \geq r} < \frac{1}{d + 1}.
\end{equation*}
Then, by the union bound, there must exist a sample $\bm{X}$ in the support of $D_0^n$ and an outcome $\wh{\mu} = \wh{\mu} (\bm{X})$ over the randomness of $\wh{\mu}$ such that:
\begin{equation*}
    \forall D \in \mc{D}: \norm{\wh{\mu} (\bm{X}) - \mu (D)}_{\Sigma (D)} \leq r.
\end{equation*}
Letting $\wh{\mu} = \wh{\mu} (\bm{X})$, we have for the direction $\bm{1}$ by \cref{lem:one_d_proj}:
\begin{equation*}
    r \geq \norm{\wh{\mu} - \mu (D_{0})}_{\Sigma (D_{0})} \geq \frac{1}{\sqrt{d} \sigma}  \abs*{\sum_{i = 1}^d \wh{\mu}_i - 1}.
\end{equation*}
Therefore, there exists $i \in [d]$ with:
\begin{equation*}
    \wh{\mu}_i \geq \frac{1 - \sqrt{d} \sigma r}{d}.
\end{equation*}
For this $i$, by considering the direction $e_i$ we have by \cref{lem:one_d_proj}:
\begin{align*}
    \norm{\wh{\mu} - \mu (D_i)}_{\Sigma (D_i)} 
    &\geq \frac{d}{\sqrt{d\eta (1 - d\eta)} + d \sigma} \cdot \lprp{\frac{1 - \sqrt{d} \sigma r}{d} - (1 - d\eta) \frac{1}{d}} \\
    &= \frac{1}{\eta/(2r) + \sigma} \cdot \lprp{\eta - \frac{\sigma r}{\sqrt{d}} } > r,
\end{align*}
where the last inequality holds with our setting of $\sigma$. This leads to a contradiction concluding the proof of the theorem.\qed

\section{Acknowledgements}

The authors would like to thank Ishaq Aden-Ali, Peter Bartlett, Abhishek Shetty, and Nikita Zhivotovskiy for numerous helpful conversations through the course of this project.

\bibliographystyle{alpha}
\bibliography{refs}

\appendix

\section{Auxiliary Results}
\label{sec:aux}

Here we list auxiliary results important to our analysis. The first is Helly's celebrated theorem \cite{helly} on convex intersections as stated in \cite[Theorem 1.1, Chapter 2.1]{handbookconvex}.

\begin{theorem}[\cite{helly,handbookconvex}]
    \label{thm:helly}
    Let $\mc{K}$ be a family of convex sets in $\R^d$, and suppose $\mc{K}$ is finite or each member of $\mc{K}$ is compact. If every $d + 1$ or fewer members of $\mc{K}$ have a common point, then there is a point common to all members of $\mc{K}$.
\end{theorem}

We also require Bousquet's famous inequality on the suprema of empirical processes \cite{bousquetthesis} which builds on prior results by Talagrand \cite{talagrandsharper,talagrandnewconcentration}.

\begin{theorem}[\cite{bousquetthesis,boucheron2013concentration}]
    \label{thm:bousequet_thm}
    Let $X_1, \dots, X_n$ be independent identically distributed random vectors indexed by an index set $\mc{T}$. Assume that $\E [X_{i, s}] = 0$, and $X_{i, s} \leq 1$ for all $s \in \mc{T}$. Let $Z = \sup_{s \in \mc{T}} \sum_{i = 1}^n X_{i, s}$, $\nu = 2 \E Z + \sigma^2$ where $\sigma^2 = \sup_{s \in \mc{T}} \sum_{i = 1}^n \E X_{i, s}^2$ is the wimpy variance. Let $\phi(u) = e^u - u - 1$ and $h(u) = (1 + u)\log (1 + u) - u$, for $u \geq -1$. Then for all $\lambda \geq 0$,
    \begin{equation*}
        \log \E e^{\lambda (Z - \E Z)} \leq \nu \phi(\lambda).
    \end{equation*}
    Also, for all $t \geq 0$,
    \begin{equation*}
        \P \lbrb{Z \geq \E Z + t} \leq e^{-\nu h(t / \nu)} \leq \exp \lprp{-\frac{t^2}{2(\nu + t / 3)}}.
    \end{equation*}
\end{theorem}

We present the well-known Ledoux-Talagrand contraction inequality \cite{ledoux1991probability}. 

\begin{theorem}[\cite{ledoux1991probability,boucheron2013concentration}]
    \label{thm:ledtal}
    Let $x_1, \dots, x_n$ be vectors whose real-valued components are indexed by $\mc{T}$, that is, $x_i = (x_{i, s})_{s \in \mc{T}}$. For each $i = 1, \dots, n$, let $\phi_i: \R \to \R$ be a $1$-Lipschitz function such that $\phi_i (0) = 0$. Let $\eps_1, \dots, \eps_n$ be independent Rademacher random variables, and let $\Psi: [0, \infty) \to \R$ by a non-decreasing convex function. Then,
    \begin{equation*}
        \E \lsrs{\Psi \lprp{\sup_{s \in \mc{T}} \sum_{i = 1}^n \eps_i \phi_i (x_{i, s})}} \leq \E \lsrs{\Psi \lprp{\sup_{s \in \mc{T}} \sum_{i = 1}^n \eps_i x_{i, s}}}
    \end{equation*}
    and 
    \begin{equation*}
        \E \lsrs{\Psi \lprp{\frac{1}{2} \sup_{s \in \mc{T}} \abs*{\sum_{i = 1}^n \eps_i \phi_i (x_{i, s})}}} \leq \E \lsrs{\Psi \lprp{\sup_{s \in \mc{T}} \abs*{\sum_{i = 1}^n \eps_i x_{i, s}}}}.
    \end{equation*}
\end{theorem}

We will use the following simple corollary of the second conclusion in our proof.

\begin{corollary}
    \label{cor:ledtal}
    Assume the setting of \cref{thm:ledtal}. Then,
    \begin{equation*}
        \E \lsrs{\sup_{s \in \mc{T}} \abs*{\sum_{i = 1}^n \eps_i \phi_i (x_{i, s})}} \leq 2 \E \lsrs{\sup_{s \in \mc{T}} \abs*{\sum_{i = 1}^n \eps_i x_{i, s}}}.
    \end{equation*}
\end{corollary}

We also require the following simple lemma frequently in our analysis.
\begin{lemma}
    \label{lem:one_d_proj}
    Let $\Sigma \in \R^{d \times d}$ and $\Sigma \succ 0$. Then, we have:
    \begin{equation*}
        \forall x, v \in \R^d, v \neq 0: \norm{x}_\Sigma \geq \frac{\abs{\inp{x}{v}}}{\sqrt{v^\top \Sigma v}}.
    \end{equation*}
\end{lemma}
\begin{proof}
    We have:
    \begin{align*}
        \norm{x}_\Sigma &= \norm*{\Sigma^{-1/2} x} \geq \abs*{\inp*{\frac{\Sigma^{1/2} v}{\norm{\Sigma^{1/2} v}}}{\Sigma^{-1/2} x}} = \frac{\abs{\inp{x}{v}}}{\sqrt{v^\top \Sigma v}}
    \end{align*}
    concluding the proof of the lemma.
\end{proof}

Next, we relate the psd ordering of two matrices to their induced Mahalanobis norms.

\begin{lemma}
    \label{lem:psd_maha}
    Let $\Sigma_1, \Sigma_2 \in \R^{d \times d}$ and $\Sigma_1 \succcurlyeq \Sigma_2 \succ 0$. Then,
    \begin{equation*}
        \forall x \in \R^d: \norm{x}_{\Sigma_1} \leq \norm{x}_{\Sigma_2}.
    \end{equation*}
\end{lemma}
\begin{proof}
    Observe
    \begin{equation*}
        \Sigma_1 \succcurlyeq \Sigma_2 \implies \Sigma_2^{-1/2} \Sigma_1 \Sigma_2^{-1/2} \succcurlyeq I \implies \Sigma_2^{1/2} \Sigma_1^{-1} \Sigma_2^{1/2} \preccurlyeq I \implies \Sigma_1^{-1} \preccurlyeq \Sigma_2^{-1}
    \end{equation*}
    yielding the lemma by the definition of $\norm{\cdot}_{\Sigma}$.
\end{proof}

\begin{lemma}
    \label{lem:mult_cond_counts}
    Let $k \in \N$ and $D$ be a distribution over $[k]$ with probabilities $\{p_i\}_{i = 1}^k$ with $p_i \geq 0$ and $\sum_{i = 1}^k p_i = 1$. Furthermore, let $n \in \N$ and $\bm{X} = \{X_1, \dots, X_n\}$ be generated i.i.d. from $D$. For any $j \in [k]$ and thresholds $\{\tau_{i}\}_{i = 1}^n$ and $\tau$ with $\tau_i, \tau > 0$, define the events
    \begin{gather*}
        A \coloneqq \lbrb{\bm{X} \in [k]^n: m_j (\bm{X}) \geq \tau_j} \\
        B \coloneqq \lbrb{\bm{X} \in [k]^n: \sum_{i = 1, i \neq j}^k \bm{1} \lbrb{m_i (\bm{X}) \leq \tau_i} \geq \tau} \\
        \text{where } m_i (\bm{X}) = \sum_{l = 1}^n \bm{1} \lbrb{X_l = i}.
    \end{gather*}
    Then, we must have:
    \begin{equation*}
        \P (B \mid A) \geq \P (B).
    \end{equation*}
\end{lemma}
\begin{proof}
    Note that
    \begin{equation*}
        \P(B \mid A) - \P (B) = \P (A^c) \cdot (\P (B \mid A) - \P (B \mid A^c)).
    \end{equation*}
    Furthermore, we have:
    \begin{gather*}
        \P (B \mid A) = \frac{1}{\sum_{t \geq \tau_j} \P (m_j (\bm{X}) = t)}\sum_{t \geq \tau_j} \P (m_j (\bm{X}) = t) \P (B \mid m_j (\bm{X}) = t) \\
        \P (B \mid A^c) = \frac{1}{\sum_{t < \tau_j} \P (m_j (\bm{X}) = t)}\sum_{t < \tau_j} \P (m_j (\bm{X}) = t) \P (B \mid m_j (\bm{X}) = t). \numberthis{} \label{eq:pb_decomp}
    \end{gather*}
    Moreover, observe by symmetry:
    \begin{gather*}
        \P (B \mid m_j (\bm{X}) = t) = \P (B \mid T_t) \\
        \text{ where } T_t \coloneqq \{\bm{X}: X_{n}, X_{n - 1} \dots, X_{n - t + 1} = j \text{ and } \forall l \leq n - t: X_{l} \neq j\} \numberthis{} \label{eq:t_m_j_equiv}
    \end{gather*}
    and that conditioned on $T$, $X_l$ for $l \leq n - t$ is a distribution over $[k] \setminus \{j\}$ with probabilities proportional to $\{p_i\}_{i \neq j}$. Next, we show for all $t_1 \geq t_2$
    \begin{equation}
        \label{eq:t_1_t_2_comp}
        \P (B \mid T_{t_1}) \geq \P (B \mid T_{t_2}).
    \end{equation} 
    Let $D_j$ denote the conditional distribution of $D$ conditioned on $[k] \setminus \{j\}$ and $Y_1, \dots Y_n$ be i.i.d. draws from $D_j$. Then define 
    \begin{equation*}
        \bm{X}^{t} = \{X^t_1, \dots, X^t_n\} \text{ with } X^t_i = 
        \begin{cases}
            Y_i &\text{if } i \leq n - t \\
            j &\text{otherwise}
        \end{cases}.
    \end{equation*}
    Note that $\bm{X}^{t}$ is identical in distribution to $\bm{X}$ conditioned on $T_t$. Furthermore, note that by definition $\bm{X}^{t_2} \in B$ implies $\bm{X}^{t_1} \in B$. Hence, \cref{eq:t_1_t_2_comp} is established. This, by \cref{eq:t_m_j_equiv}, also establishes
    \begin{equation*}
        \forall t_1 \geq t_2: \P (B \mid m_j (\bm{X}) = t_1) \geq \P (B \mid m_j (\bm{X}) = t_2)
    \end{equation*}
    which yields the lemma by \cref{eq:pb_decomp}.
\end{proof}

\end{document}